\documentclass[12pt]{article}
\usepackage{amssymb}
\usepackage{amsthm}
\usepackage{amsmath}
\usepackage{graphicx}
\usepackage{enumerate}
\usepackage{pict2e}
\usepackage{framed}

\newtheorem{theorem}{Theorem}[section]
\newtheorem{definition}[theorem]{Definition}
\newtheorem{lemma}[theorem]{Lemma}
\newtheorem{proposition}[theorem]{Proposition}

\newtheorem{remark}[theorem]{Remark}
\newtheorem{example}[theorem]{Example}

\title{Algebras and varieties where Sasaki operations form an adjoint pair}
\author{Ivan~Chajda$^1$ $\cdot$ Helmut~L\"anger$^{1,2,0}$}
\date{}

\begin{document}
	
\footnotetext{Support of the research of the first author by the Czech Science Foundation (GA\v CR), project 24-14386L, entitled ``Representation of algebraic semantics for substructural logics'', and by IGA, project P\v rF~2024~011, is gratefully acknowledged.}

\footnotetext{Corresponding author \\
Helmut L\"anger \\
helmut.laenger@tuwien.ac.at}
	
\maketitle
	
\begin{abstract}
The so-called Sasaki projection was introduced by U.~Sasaki on the lattice $\mathbf L(\mathbf H)$ of closed linear subspaces of a Hilbert space $\mathbf H$ as a projection of $\mathbf L(\mathbf H)$ onto a certain sublattice of $\mathbf L(\mathbf H)$. Since $\mathbf L(\mathbf H)$ is an orthomodular lattice, the Sasaki projection and its dual can serve as the logical connectives conjunction and implication within the logic of quantum mechanics. It was shown by the authors in their previous paper \cite{CL17} that these operations form a so-called adjoint pair. The natural question arises if this result can be extended also to lattices with a unary operation which need not be orthomodular or to other algebras with two binary and one unary operation. To show that this is possible is the aim of the present paper. We determine a variety of lattices with a unary operation where the Sasaki operations form an adjoint pair and we continue with so-called $\lambda$-lattices and certain classes of semirings. We show that the Sasaki operations have a deeper sense than originally assumed by their author and can be applied also outside the lattices of closed linear subspaces of a Hilbert space.
\end{abstract}
	
	

\section{Preliminaries}

Consider a bounded complemented lattice $\mathbf L=(L,\vee,\wedge,{}',0,1)$ where the unary operation $'$ is a {\em complementation}, i.e.\ $\mathbf L$ satisfies the identities $x\vee x'\approx1$ and $x\wedge x'\approx0$. $\mathbf L$ is called {\em orthomodular} (see \cite{Be}) if the complementation $'$ is an {\em antitone involution} and $\mathbf L$ satisfies the {\em orthomodular law}, i.e.\ the identity
\begin{enumerate}
	\item[(OM)] $x\vee\big((x\vee y)\wedge x'\big)\approx x\vee y$.
\end{enumerate}
Apparently, the class of orthomodular lattices forms a variety.

A {\em projection} is a mapping $f$ from a set $M$ to $M$ satisfying $f\circ f=f$. In such a case $f$ is called a projection from $M$ onto $f(M)$. Let $\mathbf P=(P,\le)$ be a poset, $'$ an antitone involution on $\mathbf P$ and $f\colon P\to P$. Then the {\em dual} of $f$ is the mapping $\overline f\colon P\to P$ defined by $\overline f(x):=\big(f(x')\big)'$ for all $x\in P$. If $f$ is a projection or a monotone mapping then $\overline f$ has the same property, respectively. Now let $(L,\vee,\wedge,{}',0,1)$ be an orthomodular lattice and $a\in L$. The following mapping $p_a\colon L\to L$ was introduced by U.~Sasaki \cite{Sa}, see also \cite{Be}:
\[
p_a(x):=(x\vee a')\wedge a
\]
for all $x\in L$. This mapping is a monotone projection from $L$ onto $[0,a]$ and is usually called the {\em Sasaki projection} from $L$ onto $[0,a]$. The dual $\overline{p_a}$ of $p_a$ is defined by
\[
\overline{p_a}(x):=\big(p_a(x')\big)'=\big((x'\vee a')\wedge a\big)'=(x\wedge a)\vee a'
\]
for all $x\in L$ and it is a monotone projection from $L$ onto $[a',1]$. For more information on Sasaki projections cf.\ \cite{GGN}. In what follows we will call binary operations defined in a similar way {\em Sasaki operations}.

Let $(P,\le)$ be a poset and $f,g$ binary operations on $P$. We introduce the following statements:
\begin{enumerate}
	\item[(A1)] If $f(x,y)\le z$ then $x\le g(y,z)$,
	\item[(A2)] if $x\le g(y,z)$ then $f(x,y)\le z$
\end{enumerate}
for all $x,y,z\in P$. Recall that $f$ and $g$ are said to form an {\em adjoint pair} if they satisfy both conditions (A1) and (A2). In such a case we say that $f$ and $g$ are connected via {\em adjointness}. If $f$ and $g$ form an adjoint pair then everyone of the two operations $f$ and $g$ determines the other one. Namely, for every $x,y\in P$, $f(x,y)$ is the smallest element $z$ of $P$ satisfying the inequality $x\le g(y,z)$, and for every $y,z\in P$, $g(y,z)$ is the greatest element $x$ of $P$ satisfying the inequality $f(x,y)\le z$.

It is easy to prove that if the binary operations $f$ and $g$ form an adjoint pair on a given poset $(P,\le)$ then $f$ is monotone in the first variable and $g$ in the second one.

\begin{lemma}\label{lem1}
Let $(P,\le)$ be a poset, $a,b,c\in P$ with $a\le b$ and $f,g$ binary operations on $P$ forming an adjoint pair. Then $f(a,c)\le f(b,c)$ and $g(c,a)\le g(c,b)$.
\end{lemma}

\begin{proof}
Any of the following assertions implies the next one:
\begin{align*}
f(b,c) & \le f(b,c), \\
     b & \le g\big(c,f(b,c)\big), \\
     a & \le g\big(c,f(b,c)\big), \\
f(a,c) & \le f(b,c).
\end{align*}
Moreover, any of the following assertions implies the next one:
\begin{align*}
             g(c,a) & \le g(c,a), \\
f\big(g(c,a),c\big) & \le a, \\
f\big(g(c,a),c\big) & \le b, \\
             g(c,a) & \le g(c,b).
\end{align*}
\end{proof}

The classical example of an adjoint pair are the operations $\wedge$ and $\to$ on a Boolean algebra $(B,\vee,\wedge,{}',$ $0,1)$ where $x\to y:=x'\vee y$ for all $x,y\in B$ or, more general, the operations $\wedge$ and $\to$ on a relatively pseudocomplemented meet-semilattice $(S,\wedge,*)$ where $x\to y:=x*y$ for all $x,y\in S$ and $x*y$ denotes the relative pseudocomplement of $x$ with respect to $y$. Recall that for two elements $x$ and $y$ of a meet-semilattice $(S,\wedge)$ the {\em relative pseudocomplement} of $x$ with respect to $y$ is the greatest element $z$ of $S$ satisfying $x\wedge z\le y$. The {\em meet-semilattice} is called {\em relatively pseudocomplemented} if any two of its elements have a relative pseudocomplement, see \cite{Bi} for details.

It was shown by the authors in \cite{CL17} that if $\mathbf L=(L,\vee,\wedge,{}',0,1)$ is an orthomodular lattice then the Sasaki operations on $L$ defined by the afore mentioned projections, i.e.
\begin{enumerate}
\item[(S1)] $x\odot y:=(x\vee y')\wedge y\quad$ and $\quad x\to y:=x'\vee(x\wedge y)$
\end{enumerate}
for all $x,y\in L$, form an adjoint pair. Note that for the Sasaki operations defined by (S1) we have $x\odot y=p_y(x)$ and $x\to y=\overline{p_x}(y)$ for all $x,y\in L$. In case of (S1), conditions (A1) and (A2) read as follows:
\begin{enumerate}
	\item[(A1)] If $x\odot y\le z$ then $x\le y\to z$,
	\item[(A2)] if $x\le y\to z$ then $x\odot y\le z$
\end{enumerate}
for all $x,y,z\in L$. However, such conditions hold also in the case when $\mathbf L$ is not an orthomodular lattice. Namely, in order to prove adjointness we only used (OM), but not the fact that $'$ is an antitone involution. In fact, in a modular lattice with complementation, the choice of $'$ even determines whether $\mathbf L$ is orthomodular or not. For example, consider the complemented modular lattice $\mathbf L=(L,\vee,\wedge,{}',0,1)$ depicted in Fig.~1:

\vspace*{-3mm}

\begin{center}
	\setlength{\unitlength}{7mm}
	\begin{picture}(14,8)
		\put(4,1){\circle*{.3}}
		\put(1,3){\circle*{.3}}
		\put(3,3){\circle*{.3}}
		\put(5,3){\circle*{.3}}
		\put(7,3){\circle*{.3}}
		\put(10,3){\circle*{.3}}
		\put(4,5){\circle*{.3}}
		\put(7,5){\circle*{.3}}
		\put(9,5){\circle*{.3}}
		\put(11,5){\circle*{.3}}
		\put(13,5){\circle*{.3}}
		\put(10,7){\circle*{.3}}
		\put(4,1){\line(-3,2)3}
		\put(4,1){\line(-1,2)1}
		\put(4,1){\line(1,2)1}
		\put(4,1){\line(3,2)3}
		\put(10,3){\line(-3,2)3}
		\put(10,3){\line(-1,2)1}
		\put(10,3){\line(1,2)1}
		\put(10,3){\line(3,2)3}
		\put(4,5){\line(-3,-2)3}
		\put(4,5){\line(-1,-2)1}
		\put(4,5){\line(1,-2)1}
		\put(4,5){\line(3,-2)3}
		\put(10,7){\line(-3,-2)3}
		\put(10,7){\line(-1,-2)1}
		\put(10,7){\line(1,-2)1}
		\put(10,7){\line(3,-2)3}
		\put(4,1){\line(3,1)6}
		\put(1,3){\line(3,1)6}
		\put(3,3){\line(3,1)6}
		\put(5,3){\line(3,1)6}
		\put(7,3){\line(3,1)6}
		\put(4,5){\line(3,1)6}
		\put(3.85,.3){$0$}
		\put(.4,2.85){$a$}
		\put(2.4,2.85){$b$}
		\put(5.3,2.85){$c$}
		\put(7.3,2.85){$d$}
		\put(10.3,2.85){$e$}
		\put(3.4,4.85){$f$}
		\put(6.4,4.85){$g$}
		\put(8.4,4.85){$h$}
		\put(11.3,4.85){$i$}
		\put(13.3,4.85){$j$}
		\put(9.85,7.4){$1$}
		\put(6.2,-.75){{\rm Fig.~1}}
		\put(3,-1.75){Complemented modular lattice}
	\end{picture}
\end{center}

\vspace*{10mm}

If we choose $'$ as follows:
\[
\begin{array}{l|cccccccccccc}
	x  & 0 & a & b & c & d & e & f & g & h & i & j & 1 \\
	\hline
	x' & 1 & h & i & j & g & f & e & b & c & d & a & 0
\end{array}
\]
then $\mathbf L$ is not an orthomodular lattice since $'$ is not an involution. However, also in this case one can introduce $\odot$ and $\to$ by Sasaki operations on $L$ in such a way that these operations form an adjoint pair (cf.\ Proposition~\ref{prop7}(iii)).

Hence the natural question arises when two binary operations $\odot$ and $\to$ on a set form an adjoint pair. In general, we need not consider even a complemented lattice, we ask only an algebra with two binary operations and one unary operation, for example a semiring $(S,+,\cdot,0,{}')$ with an additional unary operation $'$. We need not assume $\cdot$ to be distributive with respect to $+$, i.e.
\[
(x+y)z\approx xz+yz\text{ or }z(x+y)\approx zx+zy,
\]
but we need that a partial order relation $\le$ is defined on our algebra. An example of such an algebra may e.g.\ be a so-called $\lambda$-lattice, see \cite{CL11}. However, if the Sasaki operations on a bounded lattice with a unary operation $'$ form an adjoint pair then $'$ must be a complementation, see the following result.

\begin{lemma}
Let $\mathbf L=(L,\vee,\wedge,{}')$ be a lattice with a unary operation $'$ and $\odot$ and $\to$ denote the Sasaki operations on $L$ defined by {\rm(S1)}. Then the following holds:
\begin{enumerate}[{\rm(i)}]
\item If $\mathbf L$ has a top element $1$ and $\odot$ and $\to$ satisfy condition {\rm(A1)} then $\mathbf L$ satisfies the identity $x\vee x'\approx1$,
\item if $\mathbf L$ has a bottom element $0$ and $\odot$ and $\to$ satisfy condition {\rm(A2)} then $\mathbf L$ satisfies the identity $x\wedge x'\approx0$,
\item if $\mathbf L$ is bounded and $\odot$ and $\to$ form an adjoint pair then $'$ is a complementation on $\mathbf L$.
\end{enumerate}
\end{lemma}

\begin{proof}
Let $a\in L$.
\begin{enumerate}[(i)]
\item Because of $1\odot a\le1$ we have $1\le a\to1=a'\vee(a\wedge1)=a'\vee a$ and hence $a\vee a'=1$.
\item Because of $0\le a\to0$ we have $a'\wedge a=(0\vee a')\wedge a=0\odot a\le0$ and hence $a\wedge a'=0$.
\item This follows from (i) and (ii).
\end{enumerate}
\end{proof}

\section{Lattices}

In this section we investigate the Sasaki operations on lattices with a unary operation $'$. We are going to present some classes of lattices, in fact varieties, where the Sasaki operations form an adjoint pair.

For lattices $(L,\vee,\wedge,{}')$ with a unary operation $'$ we define the following identities:
\begin{enumerate}
\item[(B1)] $y'\vee\big((x\vee y')\wedge y\big)\approx x\vee y'$,
\item[(B2)] $\big(x'\vee(x\wedge y)\big)\wedge x\approx x\wedge y$.
\end{enumerate}
We study the Sasaki operations in the variety of lattices satisfying identities (B1) and (B2). Observe that if $'$ an antitone involution then anyone of the identities (B1) and (B2) implies the other one.

\begin{proposition}
Identities {\rm(B1)} and {\rm(B2)} are independent.
\end{proposition}

\begin{proof}
Let $(L,\vee,\wedge,0,1)$ be a non-trivial bounded lattice. If we define a unary operation $'$ on $L$ by $x':=1$ for all $x\in L$ then $(L,\vee,\wedge,{}')$ satisfies identity {\rm(B1)} since
\[
y'\vee\big((x\vee y')\wedge y\big)\approx1\vee\big((x\vee 1)\wedge y\big)\approx1\approx x\vee 1\approx x\vee y',
\]
but does not satisfy identity {\rm(B2)} since
\[
\big(1'\vee(1\wedge0)\big)\wedge1=1\vee0=1\ne0=1\wedge0.
\]
If we define a unary operation $'$ on $L$ by $x':=0$ for all $x\in L$ then $(L,\vee,\wedge,{}')$ satisfies identity {\rm(B2)} since
\[
\big(x'\vee(x\wedge y)\big)\wedge x\approx\big(0\vee(x\wedge y)\big)\wedge x\approx(x\wedge y)\wedge x\approx x\wedge y,
\]
but does not satisfy identity {\rm(B1)} since
\[
0'\vee\big((1\vee0')\wedge0\big)=0\vee0=0\ne1=1\vee0=1\vee0'.
\]
\end{proof}

The following theorem shows when the Sasaki operations $\odot$ and $\to$ satisfy condition (A1) or condition (A2), respectively, depending on the afore mentioned identities.

\begin{theorem}\label{th4}
Let $\mathbf L=(L,\vee,\wedge,{}')$ be a lattice with a unary operation $'$ and $\odot$ and $\to$ denote the Sasaki operations on $L$ defined by {\rm(S1)}. Then the following holds:
\begin{enumerate}[{\rm(i)}]
\item If $\mathbf L$ satisfies identity {\rm(B1)} then $\odot$ and $\to$ satisfy condition {\rm(A1)},
\item if $\mathbf L$ satisfies identity {\rm(B2)} then $\odot$ and $\to$ satisfy condition {\rm(A2)}
\item If $\mathbf L$ satisfies identities {\rm(B1)} and {\rm(B2)} then $\odot$ and $\to$ form an adjoint pair.
\end{enumerate}
\end{theorem}

\begin{proof}
Let $a,b,c\in A$.
\begin{enumerate}[(i)]
\item If $a\odot b\le c$ then using identity (B1) we obtain
\[
a\le a\vee b'=b'\vee\big((a\vee b')\wedge b\big)=b'\vee\Big(b\wedge\big((a\vee b')\wedge b\big)\Big)=b'\vee\big(b\wedge(a\odot b)\big)\le b'\vee(b\wedge c)=b\to c.
\]
\item If $a\le b\to c$ then using identity (B2) we obtain
\[
a\odot b=(a\vee b')\wedge b\le\big((b\to c)\vee b'\big)\wedge b=\Big(\big(b'\vee(b\wedge c)\big)\vee b'\Big)\wedge b=\big(b'\vee(b\wedge c)\big)\wedge b=b\wedge c\le c.
\]
\item This follows from (i) and (ii).
\end{enumerate}
\end{proof}

If the lattice with a unary operation is even modular, we can simplify our assumptions essentially, i.e.\ we need not assume identities (B1) and (B2) a priori, see the following result.

\begin{proposition}\label{prop7}
	Let $\mathbf L=(L,\vee,\wedge,{}')$ be a modular lattice with a unary operation $'$ and $\odot$ and $\to$ denote the Sasaki operations on $L$ defined by {\rm(S1)}. Then the following holds:
	\begin{enumerate}[{\rm(i)}]
		\item If $\mathbf L$ has a top element $1$ and satisfies the identity $x\vee x'\approx1$ then $\mathbf L$ satisfies identity {\rm(B1)} and hence $\odot$ and $\to$ satisfy condition {\rm(A1)},
		\item if $\mathbf L$ has a bottom element $0$ and satisfies the identity $x\wedge x'\approx0$ then $\mathbf L$ satisfies identity {\rm(B2)} and hence	$\odot$ and $\to$ satisfy condition {\rm(A2)},
		\item if $\mathbf L$ is complemented then $\odot$ and $\to$ form an adjoint pair.
	\end{enumerate}
\end{proposition}

\begin{proof}
	\
	\begin{enumerate}[(i)]
	\item Assume $\mathbf L$ to have a top element $1$ and to satisfy the identity $x\vee x'\approx1$. Then
	\[
	y'\vee\big((x\vee y')\wedge y\big)\approx y'\vee\big(y\wedge(x\vee y')\big)\approx(y'\vee y)\wedge(x\vee y')\approx1\wedge(x\vee y')\approx x\vee y'
	\]
	and hence $\mathbf L$ satisfies identity (B1) and therefore $\odot$ and $\to$ satisfy condition (A1) according to Theorem~\ref{th4} (i).
	\item Assume $\mathbf L$ to have a bottom element $0$ and to satisfy the identity $x\wedge x'\approx0$. Then
	\[
	\big(x'\vee(x\wedge y)\big)\wedge x\approx\big((x\wedge y)\vee x'\big)\wedge x\approx(x\wedge y)\vee(x'\wedge x)\approx(x\wedge y)\vee0\approx x\wedge y
	\]
	and hence $\mathbf L$ satisfies identity (B2) and therefore $\odot$ and $\to$ satisfy condition (A2) according to Theorem~\ref{th4} (ii).
	\item This follows from (i) and (ii).
	\end{enumerate}
\end{proof}

Concerning Proposition~\ref{prop7} we make the following remark.

\begin{remark}
Recall that a {\em meet-semilattice} $(S,\wedge,0)$ with bottom element $0$ is called {\em pseudocomplemented} if for every $x\in S$ there exists a greatest element $x^*$ of $S$ satisfying $x\wedge x^*=0$. It is clear that every finite distributive lattice is pseudocomplemented. Hence we can apply Proposition~\ref{prop7} {\rm(ii)} to finite distributive lattices in order to see that the Sasaki operations defined by {\rm(S1)} satisfy condition {\rm(A2)}. Recall that a {\em join-semilattice} $(S,\vee,1)$ with top element $1$ is called {\em dually pseudocomplemented} if for every $x\in S$ there exists a smallest element $x^d$ of $S$ satisfying $x\vee x^d=1$. It is clear that every finite distributive lattice is dually pseudocomplemented. Hence we can apply Proposition~\ref{prop7} {\rm(i)} to finite distributive lattices in order to see that the Sasaki operations defined by {\rm(S1)} satisfy condition {\rm(A1)}.
\end{remark}

For the next result, recall the following concepts.

A {\em lattice} $(L,\vee,\wedge,{}')$ with a unary operation $'$ is called {\em weakly orthomodular} respectively {\em dually weakly orthomodular} (cf.\ \cite{CL18}) if it satisfies the identity
\begin{align*}
	x & \approx(x\wedge y)\vee\big(x\wedge(x\wedge y)'\big)\text{ or} \\
	x &\approx(x\vee y)\wedge\big(x\vee(x\vee y)'\big),
\end{align*}
respectively. Hence, weakly orthomodular as well as dually weakly orthomodular lattices form a variety. Let us note that the unary operation $'$ need neither be a complementation nor an antitone involution.

\begin{proposition}\label{prop8}
Let $\mathbf L=(L,\vee,\wedge,{}')$ be a lattice with a unary operation $'$ and $\odot$ and $\to$ denote the Sasaki operations on $L$ defined by {\rm(S1)}. Then the following holds:
\begin{enumerate}[{\rm(i)}]
\item If $\mathbf L$ is weakly orthomodular and $'$ is an involution then $\mathbf L$ satisfies identity {\rm(B1)} and hence $\odot$ and $\to$ satisfy condition {\rm(A1)},
\item if $\mathbf L$ is dually weakly orthomodular then $\mathbf L$ satisfies identity {\rm(B2)} and hence $\odot$ and $\to$ satisfy condition {\rm(A1)},
\item if $\mathbf L$ is orthomodular then $\odot$ and $\to$ form an adjoint pair.
\end{enumerate}
\end{proposition}

\begin{proof}
	\
	\begin{enumerate}[(i)]
		\item Assume $\mathbf L$ to be weakly orthomodular and $'$ to be an involution. Then
		\[
		y'\vee\big((x\vee y')\wedge y\big)\approx y'\vee\big((x\vee y')\wedge y''\big)\approx x\vee y'
		\]
		and hence $\mathbf L$ satisfies identity (B1) and therefore $\odot$ and $\to$ satisfy condition (A1) according to Theorem~\ref{th4} (i).
		\item Assume $\mathbf L$ to be dually weakly orthomodular. Then
\[
\big(x'\vee(x\wedge y)\big)\wedge x\approx x\wedge\big((x\wedge y)\vee x'\big)\approx x\wedge y
\]
and hence $\mathbf L$ satisfies identity (B2) and therefore $\odot$ and $\to$ satisfy condition (A2) according to Theorem~\ref{th4} (ii).
\item This follows from (i) and (ii).
\end{enumerate}
\end{proof}

However, a lattice satisfying identity (B1) need not be modular, see the following example.

\begin{example}
	Consider the non-modular lattice $\mathbf N_5=(N_5,\vee,\wedge)$ visualized in Fig.~2:
	
	\vspace*{-3mm}
	
	\begin{center}
		\setlength{\unitlength}{7mm}
		\begin{picture}(4,8)
			\put(2,1){\circle*{.3}}
			\put(1,4){\circle*{.3}}
			\put(3,3){\circle*{.3}}
			\put(3,5){\circle*{.3}}
			\put(2,7){\circle*{.3}}
			\put(2,1){\line(-1,3)1}
			\put(2,1){\line(1,2)1}
			\put(3,3){\line(0,1)2}
			\put(2,7){\line(-1,-3)1}
			\put(2,7){\line(1,-2)1}
			\put(1.85,.3){$0$}
			\put(3.4,2.85){$a$}
			\put(.35,3.85){$b$}
			\put(3.4,4.85){$c$}
			\put(1.85,7.4){$1$}
			\put(1.2,-.75){{\rm Fig.~2}}
			\put(-1.1,-1.75){Non-modular lattice $\mathbf N_5$}
		\end{picture}
	\end{center}
	
	\vspace*{8mm}
	
	where the complementation $'$ is defined as follows:
	\[
	\begin{array}{l|ccccc}
		x  & 0 & a & b  & c & 1 \\
		\hline
		x' & 1 & b & b' & b & 0
	\end{array}
	\]
	with $b'\in\{a,c\}$. Then $'$ is not an involution since $c''=b'=a\ne c$ in case $b'=a$ and $a''=b'=c\ne a$ in case $b'=c$. Abbreviate $(N_5,\vee,\wedge,{}')$ by $\mathbf N_5'$ and let $\odot$ and $\to$ denote the Sasaki operations on $N_5$ defined by {\rm(S1)}. In case $b'=c$ the algebra $\mathbf N_5'$ satisfies identity {\rm(B1)} and hence also condition {\rm(A1)}. In case $b'=a$ the algebra $\mathbf N_5'$ does not satisfy condition {\rm(A1)} since
	\[
	c\odot b=(c\vee b')\wedge b=(c\vee a)\wedge b=c\wedge b=0,
	\]
	but
	\[
	c\not\le a=a\vee0=b'\vee(b\wedge0)=b\to0.
	\]
	and hence $\mathbf N_5'$ does not satisfy identity {\rm(B1)}. In any case, $\mathbf N_5'$ does not satisfy condition {\rm(A2)} since
	\[
	a\le1=b\vee a=c'\vee(c\wedge a)=c\to a,
	\]
	but
	\[
	a\odot c=(a\vee c')\wedge c=(a\vee b)\wedge c=1\wedge c=c\not\le a
	\]
	and hence $\mathbf N_5'$ does not satisfy identity {\rm(B2)}.
\end{example}

\section{$\lambda$-lattices}

Other ordered algebras with two binary and one unary operation where the Sasaki operations can be studied are the so-called $\lambda$-lattices.

For every poset $(P,\le)$ and any $a,b\in P$ we define the upper cone $U(a,b)$ of $a$ and $b$ by
\[
U(a,b):=\{x\in A\mid a\le x\text{ and }b\le x\}
\]
 and lower cone $L(a,b)$ of $a$ and $b$ by
\[
L(a,b):=\{x\in A\mid x\le a\text{ and }x\le b\}.
\]
Let us recall the concept of a $\lambda$-lattice introduced by V.~Sn\'a\v sel \cite{Sn}, see also \cite{CL11}. A {\em $\lambda$-lattice} is an algebra $(A,\sqcup,\sqcap)$ of type $(2,2)$ satisfying the following identities:
\[
\begin{array}{ll}
	x\sqcup y\approx y\sqcup x,                                     & x\sqcap y\approx y\sqcap x, \\
	x\sqcup\big((x\sqcup y)\sqcup z\big)\approx(x\sqcup y)\sqcup z, & x\sqcap\big((x\sqcap y)\sqcap z\big)\approx(x\sqcap y)\sqcap z, \\
	x\sqcup(x\sqcap y)\approx x,                                    & x\sqcap(x\sqcup y)\approx x.
\end{array}
\]
It is evident that the class of $\lambda$-lattices forms a variety. It is immediate to check that it satisfies the idempotent laws
\[
x\sqcup x\approx x\text{ and }x\sqcap x\approx x.
\]
In a $\lambda$-lattice a partial order relation $\le$, the so-called {\em induced order}, can be introduced by
\[
x\le y\text{ if and only if }x\sqcup y=y\text{ if and only if }x\sqcap y=x
\]
($x,y\in A$), see \cite{CL11} for details. Every poset $(A,\le)$ having the property that any two elements have at least one lower bound and at least one upper bound can be converted into a $\lambda$-lattice by defining binary operations $\sqcup$ and $\sqcap$ as follows: \\
If $a\le b$ then $a\sqcup b=b\sqcup a:=b$ and $a\sqcap b=b\sqcap a:=a$. \\
If $a\parallel b$ then $a\sqcup b=b\sqcup a$ is an arbitrary element of $U(a,b)$, and $a\sqcap b=b\sqcap a$ is an arbitrary element of $L(a,b)$. It is elementary to verify the identities of a $\lambda$-lattice. Of course, every lattice is a $\lambda$-lattice, but not vice versa. Fig.~3 shows a $\lambda$-lattice that is not a lattice:

\vspace*{-5mm}

\begin{center}
	\setlength{\unitlength}{7mm}
	\begin{picture}(4,8)
		\put(2,1){\circle*{.3}}
		\put(1,3){\circle*{.3}}
		\put(3,3){\circle*{.3}}
		\put(1,5){\circle*{.3}}
		\put(3,5){\circle*{.3}}
		\put(2,7){\circle*{.3}}
		\put(1,3){\line(0,1)2}
		\put(1,3){\line(1,1)2}
		\put(1,3){\line(1,-2)1}
		\put(3,3){\line(-1,-2)1}
		\put(3,3){\line(-1,1)2}
		\put(3,3){\line(0,1)2}
		\put(2,7){\line(-1,-2)1}
		\put(2,7){\line(1,-2)1}
		\put(1.85,.3){$0$}
		\put(.35,2.85){$a$}
		\put(3.4,2.85){$b=c\sqcap d$}
		\put(.35,4.85){$c$}
		\put(3.4,4.85){$d=a\sqcup b$}
		\put(1.85,7.4){$1$}
		\put(1.2,-.75){{\rm Fig.~3}}
		\put(.6,-1.75){A $\lambda$-lattice}
	\end{picture}
\end{center}

\vspace*{10mm}

For $\lambda$-lattices $(A,\sqcup,\sqcap,{}')$ with a unary operation $'$ we introduce the following identities and conditions (which could be rewritten in the form of identities) being variants of the identities (B1) and (B2) from the previous section:
\begin{enumerate}
\item[(C1)] $y'\sqcup\big((x\sqcup y')\sqcap y\big)\approx x\sqcup y'$,
\item[(C2)] $\big(x'\sqcup(x\sqcap y)\big)\sqcap x\approx x\sqcap y$
\end{enumerate}
for all $x,y\in A$, or, in the form of inequalities,
\begin{enumerate}
\item[(D1)] $x\sqcup y'\le y'\sqcup\big((x\sqcup y')\sqcap y\big)$,
\item[(D2)] $\big(x'\sqcup(x\sqcap y)\big)\sqcap x\le x\sqcap y$
\end{enumerate}
for all $x,y\in A$. Obviously, identity (C1) implies condition (D1) and identity (C2) implies condition (D2). In a $\lambda$-lattices $(A,\sqcup,\sqcap,{}')$ with a unary operation, the Sasaki operations can be defined by
\begin{enumerate}
\item[(S2)] $x\odot y:=(x\sqcup y')\sqcap y\quad$ and $\quad x\to y:=x'\sqcup(x\sqcap y)$
\end{enumerate}
for all $x,y\in A$.

Similarly as in the case of lattices, we can prove the following result.

\begin{lemma}\label{lem2}
	Let $\mathbf A=(A,\sqcup,\sqcap,{}')$ be a $\lambda$-lattice with a unary operation $'$ and $\odot$ and $\to$ denote the Sasaki operations on $A$ defined by {\rm(S2)}. Then the following holds:
	\begin{enumerate}[{\rm(i)}]
		\item If $\mathbf A$ has a top element $1$ and $\odot$ and $\to$ satisfy condition {\rm(A1)} then $\mathbf A$ satisfies the identity $x\sqcup x'\approx1$,
		\item if $\mathbf A$ has a bottom element $0$ and $\odot$ and $\to$ satisfy condition {\rm(A2)} then $\mathbf A$ satisfies the identity $x\sqcap x'\approx0$,
		\item if $\mathbf A$ is bounded and $\odot$ and $\to$ form an adjoint pair then $'$ is a complementation on $\mathbf A$.
	\end{enumerate}
\end{lemma}

\begin{proof}
	Let $a\in A$.
	\begin{enumerate}[(i)]
		\item Because of $1\odot a\le1$ we have $1\le a\to1=a'\sqcup(a\sqcap1)=a'\sqcup a$ and hence $a\sqcup a'=1$.
		\item Because of $0\le a\to0$ we have $a'\sqcap a=(0\sqcup a')\sqcap a=0\odot a\le0$ and hence $a\sqcap a'=0$.
		\item This follows from (i) and (ii).
	\end{enumerate}
\end{proof}

Consider the following bounded poset $\mathbf P=(A,\le,{}',0,1)$ with involution $'$:

\vspace*{-5mm}

\begin{center}
	\setlength{\unitlength}{7mm}
	\begin{picture}(8,8)
		\put(4,1){\circle*{.3}}
		\put(1,3){\circle*{.3}}
		\put(3,3){\circle*{.3}}
		\put(5,3){\circle*{.3}}
		\put(7,3){\circle*{.3}}
		\put(1,5){\circle*{.3}}
		\put(3,5){\circle*{.3}}
		\put(5,5){\circle*{.3}}
		\put(7,5){\circle*{.3}}
		\put(4,7){\circle*{.3}}
		\put(4,1){\line(-3,2)3}
		\put(4,1){\line(-1,2)1}
		\put(4,1){\line(1,2)1}
		\put(4,1){\line(3,2)3}
		\put(1,3){\line(0,1)2}
		\put(1,3){\line(1,1)2}
		\put(1,3){\line(2,1)4}
		\put(3,3){\line(-1,1)2}
		\put(3,3){\line(2,1)4}
		\put(5,3){\line(-2,1)4}
		\put(5,3){\line(1,1)2}
		\put(7,3){\line(-2,1)4}
		\put(7,3){\line(-1,1)2}
		\put(7,3){\line(0,1)2}
		\put(4,7){\line(-3,-2)3}
		\put(4,7){\line(-1,-2)1}
		\put(4,7){\line(1,-2)1}
		\put(4,7){\line(3,-2)3}
		\put(3.85,.3){$0$}
		\put(.35,2.85){$a$}
		\put(2.35,2.85){$b$}
		\put(5.4,2.85){$c$}
		\put(7.4,2.85){$d$}
		\put(.35,4.85){$d'$}
		\put(2.35,4.85){$c'$}
		\put(5.4,4.85){$b'$}
		\put(7.4,4.85){$a'$}
		\put(3.85,7.4){$1=0'$}
		\put(3.2,-.75){{\rm Fig.~4}}
		\put(0,-1.75){{\rm A bounded poset with involution}}
	\end{picture}
\end{center}

\vspace*{10mm}

This poset $\mathbf P$ can be converted into a $\lambda$-lattice in several ways. The converse of Lemma~\ref{lem2} (iii) does not hold, see the following $\lambda$-lattice. If $(A,\sqcup,\sqcap)$ is a $\lambda$-lattice with involution corresponding to $\mathbf P$ then the Sasaki operations $\odot$ and $\to$ on $A$ defined by (S2) do not form an adjoint pair, independent from the fact how $\sqcup$ and $\sqcap$ are defined within this $\lambda$-lattice. Suppose $\odot$ and $\to$ form an adjoint pair. Then we have
\begin{align*}
& b\le a\sqcup b=a\sqcup(a'\sqcap b)=a'\to b\text{ and hence }b\odot a'\le b, \\
& b\le a\sqcup c=a\sqcup(a'\sqcap c)=a'\to c\text{ and hence }b\odot a'\le c, \\
& b\not\le a=a\sqcup0=a\sqcup(a'\sqcap0)=a'\to0\text{ and hence }b\odot a'\not\le0,\text{ i.e. }b\odot a'\ne0
\end{align*}
which is a contradiction.

But there is another essential difference from the case of lattices. It is known (see e.g.\ \cite{CL11}) that the $\lambda$-lattice operations $\sqcup$ and $\sqcap$ need not be compatible with the induced order. For example we have $a\le c$ in the $\lambda$-lattice depicted in Fig.~3, but $a\sqcup b=d\not\le c=c\sqcup b$. Moreover, a $\lambda$-lattice is a lattice if and only if $\sqcup$ and $\sqcap$ are compatible with the induced order, see e.g.\ \cite{Sn} or Theorem~2.14 in \cite{CL11}. We consider a weaker version of compatibility expressed by the following conditions (E1) and (E2). These conditions are not trivial, they do not imply that the $\lambda$-lattice in question is a lattice.

For $\lambda$-lattices $(A,\sqcup,\sqcap,{}')$ with a unary operation we define the following conditions (which could be rewritten in the form of identities):
\begin{enumerate}
\item[(E1)] $x\le y$ implies $z'\sqcup(z\sqcap x)\le z'\sqcup(z\sqcap y)$,
\item[(E2)] $x\le y$ implies $(x\sqcup z')\sqcap z\le(y\sqcup z')\sqcap z$.
\end{enumerate}
Moreover, we consider also a weaker version of these conditions, namely
\begin{enumerate}
\item[(F1)] $x\odot y\le z$ implies $y'\sqcup\big(y\sqcap(x\odot y)\big)\le y'\sqcup(y\sqcap z)$,
\item[(F2)] $x\le y\to z$ implies $(x\sqcup y')\sqcap y\le\big((y\to z)\sqcup y'\big)\sqcap y$
\end{enumerate}
for all $x,y,z\in A$. Obviously, condition (E1) implies condition (F1) and condition (E2) implies condition (F2).

\begin{proposition}\label{prop1}
Let $\mathbf A=(A,\sqcup,\sqcap,{}')$ be a $\lambda$-lattice with a unary operation $'$ and $\odot$ and $\to$ denote the Sasaki operations on $A$ defined by {\rm(S2)}. Then the following holds:
\begin{enumerate}[{\rm(i)}]
\item If $\mathbf A$ satisfies conditions {\rm(D1)} and {\rm(F1)} then $\odot$ and $\to$ satisfy condition {\rm(A1)},
\item if $\mathbf A$ satisfies conditions {\rm(D2)} and {\rm(F2)} then $\odot$ and $\to$ satisfy condition {\rm(A2)}.
\end{enumerate}
\end{proposition}

\begin{proof}
Let $a,b,c\in A$.
\begin{enumerate}[(i)]
\item If $\mathbf A$ satisfies conditions (D1) and (F1) and $a\odot b\le c$ then we obtain
\begin{align*}
a & \le a\sqcup b'\le b'\sqcup\big((a\sqcup b')\sqcap b\big)=b'\sqcup\Big(b\sqcap\big((a\sqcup b')\sqcap b\big)\Big)=b'\sqcup\big(b\sqcap(a\odot b)\big)\le \\
  & \le b'\sqcup(b\sqcap c)=b\to c.
\end{align*}
\item If $\mathbf A$ satisfies conditions (D2) and (F2) and $a\le b\to c$ then we obtain
\begin{align*}
a\odot b & =(a\sqcup b')\sqcap b\le\big((b\to c)\sqcup b'\big)\sqcap b=\Big(\big(b'\sqcup(b\sqcap c)\big)\sqcup b'\Big)\sqcap b= \\
         & =\big(b'\sqcup(b\sqcap c)\big)\sqcap b\le b\sqcap c\le c.
\end{align*}
\end{enumerate}
\end{proof}

In the next example we present a $\lambda$-lattice whose Sasaki operations defined by (S2) satisfy condition (A1), but not condition (A2).

\begin{example}
	Let $\mathbf A=(A,\sqcup,\sqcap,{}')$ denote the $\lambda$-lattice from Fig.~3 with the unary operation $'$ defined by
	\[
	\begin{array}{l|cccccc}
		x  & 0 & a & b & c & d & 1 \\
		\hline
		x' & 1 & 1 & 1 & d & c & 0
	\end{array}
	\]
	and $\odot$ and $\to$ denote the Sasaki operations on $A$ defined by {\rm(S2)}. If $y\ne c,d$ then condition {\rm(D1)} clearly holds. If $y=c$ then condition {\rm(D1)} reads $d\sqcup\big((x\sqcup d)\sqcap c\big)\approx x\sqcup d$ which holds since $x\sqcup d\ge d$. If, finally, $y=d$ then condition {\rm(D1)} reads $c\sqcup\big((x\sqcup c)\sqcap d\big)\approx x\sqcup c$ which holds since $x\sqcup c\ge c$. Now assume $x,y,z\in A$ and $x\le y$. If $z\ne c,d$ then clearly
	\begin{enumerate}
		\item[{\rm(3)}] $z'\sqcup(z\sqcap x)\le z'\sqcup(z\sqcap y)$.
	\end{enumerate}
	If $z=c$ then {\rm(3)} holds since
	\[
	z'\sqcup(z\sqcap x)=d\sqcup(c\sqcap x)=\left\{
	\begin{array}{ll}
		d & \text{if }x\le d, \\
		1 & \text{otherwise}.
	\end{array}
	\right.
	\]
	If, finally, $z=d$ then {\rm(3)} holds since
	\[
	z'\sqcup(z\sqcap x)=c\sqcup(d\sqcap x)=\left\{
	\begin{array}{ll}
		c & \text{if }x\le c, \\
		1 & \text{otherwise}.
	\end{array}
	\right.
	\]
	Hence $\mathbf A$ satisfies also condition {\rm(E1)} and by Proposition~\ref{prop1} {\rm(i)} $\odot$ and $\to$ satisfy condition {\rm(A1)}. But $\odot$ and $\to$ do not satisfy condition {\rm(A2)} since
	\[
	0\le d=d\sqcup0=c'\sqcup(c\sqcap0)=c\to0,
	\]
	but
	\[
	0\odot c=(0\sqcup c')\sqcap c=d\sqcap c=b\not\le0.
	\]
	However, the unary operation $'$ on $A$ cannot be defined in such a way that both condition {\rm(D2)} and condition {\rm(E2)} are satisfied. This can be seen as follows: Suppose there exists some unary operation $'$ on $A$ satisfying both condition {\rm(D2)} and condition {\rm(E2)}. Putting $x=c$ and $y=0$ in condition {\rm(D2)} yields
	\[
	c'\sqcap c=(c'\sqcup0)\sqcap c=\big(c'\sqcup(c\sqcap0)\big)\sqcap c\le c\sqcap0=0
	\]
	whence $c'=0$. Because of $a\le d$ we have according to condition {\rm(E2)}
	\[
	a=(a\sqcup c')\sqcap c\le(d\sqcup c')\sqcap c=d\sqcap c=b,
	\]
	a contradiction.
\end{example}

We now present an example of a $\lambda$-lattice whose Sasaki operations defined by (S2) satisfy condition (A2), but not condition (A1).

\begin{example}\label{ex1}
	Consider the following bounded poset $\mathbf P=(A,\le,{}',0,1)$ with involution $'$:
	
\vspace*{-5mm}

\begin{center}
\setlength{\unitlength}{7mm}
\begin{picture}(8,8)
\put(4,1){\circle*{.3}}
\put(1,3){\circle*{.3}}
\put(3,3){\circle*{.3}}
\put(5,3){\circle*{.3}}
\put(7,3){\circle*{.3}}
\put(1,5){\circle*{.3}}
\put(3,5){\circle*{.3}}
\put(5,5){\circle*{.3}}
\put(7,5){\circle*{.3}}
\put(4,7){\circle*{.3}}
\put(4,1){\line(-3,2)3}
\put(4,1){\line(-1,2)1}
\put(4,1){\line(1,2)1}
\put(4,1){\line(3,2)3}
\put(1,3){\line(0,1)2}
\put(1,3){\line(1,1)2}
\put(1,3){\line(2,1)4}
\put(3,3){\line(-1,1)2}
\put(3,3){\line(0,1)2}
\put(3,3){\line(2,1)4}
\put(5,3){\line(-2,1)4}
\put(5,3){\line(0,1)2}
\put(5,3){\line(1,1)2}
\put(7,3){\line(-2,1)4}
\put(7,3){\line(-1,1)2}
\put(7,3){\line(0,1)2}
\put(4,7){\line(-3,-2)3}
\put(4,7){\line(-1,-2)1}
\put(4,7){\line(1,-2)1}
\put(4,7){\line(3,-2)3}
\put(3.85,.3){$0$}
\put(.35,2.85){$a$}
\put(2.35,2.85){$b$}
\put(5.4,2.85){$c$}
\put(7.4,2.85){$d$}
\put(0.35,4.85){$d'$}
\put(2.35,4.85){$c'$}
\put(5.4,4.85){$b'$}
\put(7.4,4.85){$a'$}
\put(3.85,7.4){$1=0'$}
\put(3.2,-.75){{\rm Fig.~5}}
\put(0,-1.75){{\rm A bounded poset with involution}}
\end{picture}
\end{center}

\vspace*{10mm}
	
Define a bounded $\lambda$-lattice $\mathbf A=(A,\sqcup,\sqcap,{}',0,1)$ with involution corresponding to $\mathbf P$ in the following way: Put $B:=\{a,b,c,d\}$ and $B':=\{a',b',c',d'\}$ and for different $x,y\in B$ assume $x\sqcup y\in B'\setminus\{x',y'\}$ and put $x'\sqcap y':=0$. Let $\odot$ and $\to$ denote the Sasaki operations on $A$ defined by {\rm(S2)}. Evidently, $'$ is a complementation as well as an antitone involution. We show that for all $x,y,z\in A$ we have
	\begin{enumerate}
		\item[{\rm(A2)}] $x\le y\to z=y'\sqcup(y\sqcap z)$ implies $x\odot y=(x\sqcup y')\sqcap y\le z$.
	\end{enumerate}
	First observe that $\mathbf A$ satisfies the identities $x\sqcup x'\approx1$ and $x\sqcap x'\approx0$. If $x=0$ or $y\in\{0,1\}$ then condition {\rm(A2)} holds. If $x=1$ and $x\le y'\sqcup(y\sqcap z)$ then $y'\sqcup(y\sqcap z)=1$ and hence $y\sqcap z=y$, i.e.\ $y\le z$ showing $(x\sqcup y')\sqcap y=y\le z$. Now let $e,f$ be different elements of $B$. \\
	If $(x,y)=(e,f)$ then $(x\sqcup y')\sqcap y=(e\sqcup f')\sqcap f=0\le z$, \\
	if $(x,y)=(e,e')$ then $(x\sqcup y')\sqcap y=(e\sqcup e)\sqcap e'=0\le z$, \\
	if $(x,y)=(e,f')$ then $(x\sqcup y')\sqcap y=(e\sqcup f)\sqcap f'=0\le z$, \\
	if $(x,y)=(e',e)$ then $(x\sqcup y')\sqcap y=(e'\sqcup e')\sqcap e=0\le z$, \\
	if $(x,y)=(e',f')$ then $(x\sqcup y')\sqcap y=(e'\sqcup f)\sqcap f'=0\le z$, \\
	if $(x,y)=(e,e)$ and $x\le y'\sqcup(y\sqcap z)$ then $e\le e'\sqcup(e\sqcap z)$ and hence $e\sqcap z=e$, i.e.\ $e\le z$ showing $(x\sqcup y')\sqcap y=(e\sqcup e')\sqcap e=e\le z$, \\
	if $(x,y)=(e',f)$ and $x\le y'\sqcup(y\sqcap z)$ then $e'\le f'\sqcup(f\sqcap z)$ and hence $f\sqcap z=f$, i.e.\ $f\le z$ showing $(x\sqcup y')\sqcap y=(e'\sqcup f')\sqcap f=f\le z$, \\
	if, finally, $(x,y)=(e',e')$ and $x\le y'\sqcup(y\sqcap z)$ then $e'\le e\sqcup(e'\sqcap z)$ and hence $e'\sqcap z=e'$, i.e.\ $e'\le z$ showing $(x\sqcup y')\sqcap y=(e'\sqcup e)\sqcap e'=e'\le z$. This shows that $\odot$ and $\to$ satisfy condition {\rm(A2)}. However, $\odot$ and $\to$ do not satisfy condition {\rm(A1)} since
	\[
	a\odot c'=(a\sqcup c)\sqcap c'\in\{b'\sqcap c',d'\sqcap c'\}=\{0\}
	\]
	and hence $a\odot c'\le0$, but
	\[
	a\not\le c=c\sqcup0=c\sqcup(c'\sqcap0)=c'\to0.
	\]
	Moreover, $\mathbf A$ does not satisfy identity {\rm(C2)} since by putting $(x,y)=(a',b)$ we obtain
	\[
	\big(x'\sqcup(x\sqcap y)\big)\sqcap x=\big(a\sqcup(a'\sqcap b)\big)\sqcap a'=(a\sqcup b)\sqcap a'=0\ne b=a'\sqcap b=x\sqcap y.
	\]
\end{example}

\begin{remark}
	As shown in Example~\ref{ex1}, identity {\rm(C2)} is not a necessary condition for Sasaki operations in a $\lambda$-lattice to satisfy condition {\rm(A2)}.
\end{remark}

We are going to derive a characterization of $\lambda$-lattices satisfying conditions (D1) and (D2) in which the Sasaki operations defined by (S2) form an adjoint pair.

\begin{theorem}\label{th3}
Let $\mathbf A=(A,\sqcup,\sqcap,{}')$ be a $\lambda$-lattice with a unary operation $'$ satisfying conditions {\rm(D1)} and {\rm(D2)} and $\odot$ and $\to$ denote the Sasaki operations on $A$ defined by {\rm(S2)}. Then the following are equivalent:
\begin{enumerate}[{\rm(i)}]
\item The operations $\odot$ and $\to$ form an adjoint pair,
\item the $\lambda$-lattice $\mathbf A$ satisfies conditions {\rm(E1)} and {\rm(E2)},
\item the $\lambda$-lattice $\mathbf A$ satisfies conditions {\rm(F1)} and {\rm(F2)}.
\end{enumerate}
\end{theorem}

\begin{proof}
Let $a,b,c\in A$. \\
(i) $\Rightarrow$ (ii): \\
If $a\le b$ then according to Lemma~\ref{lem1} we have
\begin{align*}
 c'\sqcup(c\sqcap a) & =c\to a\le c\to b=c'\sqcup(c\sqcap b), \\
(a\sqcup c')\sqcap c & =a\odot c\le b\odot c=(b\sqcup c)\sqcap c.
\end{align*}
(ii) $\Rightarrow$ (iii): \\
This is clear. \\
(iii) $\Rightarrow$ (i): \\
This follows from Proposition~\ref{prop1}.
\end{proof}

A $\lambda$-lattice as described in Theorem~\ref{th3} which is not a lattice is presented in Example~\ref{ex2} below.

It is worth noticing that we do not know an example of a $\lambda$-lattice (with a unary operation $'$) not being a lattice, but satisfying identities (C1) and (C2) whose Sasaki operations $\odot$ and $\to$ defined by (S2) form an adjoint pair. This indicates that the identities (C1) and (C2) are too strong for $\lambda$-lattices despite the fact that their lattice versions work well for lattices, see the previous section. This can be explained by the next theorem showing that a $\lambda$-lattice satisfying identities (C1) and (C2) where $\odot$ and $\to$ form an adjoint pair is very close to a lattice.

\begin{theorem}\label{th5}
Let $\mathbf A=(A,\sqcup,\sqcap,{}')$ be a $\lambda$-lattice with a surjective unary operation $'$ satisfying identities {\rm(C1)} and {\rm(C2)} and assume that the Sasaki operations on $A$ defined by {\rm(S2)} form an adjoint pair. Then $\mathbf A$ is a lattice.
\end{theorem}

\begin{proof}
Let $a,b,c\in A$. Since $'$ is surjective there exists some $d\in A$ with $d'=b$. Using identities (C1) and (C2) as well as Lemma~\ref{lem1} we obtain
\begin{align*}
a\sqcup c & =a\sqcup d'=d'\sqcup\big((a\sqcup d')\sqcap d\big)=d'\sqcup\Big(d\sqcap\big((a\sqcup d')\sqcap d\big)\Big)= \\
          & =d\to(a\odot d)\le d\to(b\odot d)=d'\sqcup\Big(d\sqcap\big((b\sqcup d')\sqcap d\big)\Big)=d'\sqcup\big((b\sqcup d')\sqcap d\big)= \\
          & =b\sqcup d'=b\sqcup c, \\
c\sqcap a & =\big(c'\sqcup(c\sqcap a)\big)\sqcap c=\Big(\big(c'\sqcup(c\sqcap a)\big)\sqcup c'\Big)\sqcap c=(c\to a)\odot c\le(c\to b)\odot c= \\
          & =\Big(\big(c'\sqcup(c\sqcap ab\big)\sqcup c'\Big)\sqcap c=\big(c'\sqcup(c\sqcap b)\big)\sqcap c=c\sqcap b.
\end{align*}
This means that $\sqcup$ and $\sqcap$ are monotone and according to Theorem~2.14 in \cite{CL11}, $\mathbf A$ is a lattice.
\end{proof}

We now present a $\lambda$-lattice whose Sasaki operations satisfy identities (C1) and (C2), but neither condition (A1), nor condition (A2).

\begin{example}
Put $B:=\{a,b,c,d,e,f,g\}$ and $B':=\{x'\mid x\in B\}$, assume $B$, $B'$ and $\{0,1\}$ to be pairwise disjoint and put
\[
C:=\{\{a,b,c\},\{a,d,f\},\{a,e,g\},\{b,d,g\},\{b,e,f\},\{c,d,e\},\{c,f,g\}\}.
\]
Then to any two different elements $x$ and $y$ of $B$ there exists exactly one element $F(x,y)$ of $B$ satisfying $\{x,y,F(x,y)\}\in C$. {\rm(}$C$ is the set of all ``lines'' of the Fano plane with the set $B$ of points. We illustrate this construction by the following diagram:

\vspace*{-5mm}

\begin{center}
	\setlength{\unitlength}{7mm}
	\begin{picture}(6,6)
		\put(1,1){\circle*{.3}}
		\put(3,1){\circle*{.3}}
		\put(5,1){\circle*{.3}}
		\put(2,3){\circle*{.3}}
		\put(3,2.33){\circle*{.3}}
		\put(4,3){\circle*{.3}}
		\put(3,5){\circle*{.3}}
		\put(1,1){\line(1,0)4}
		\put(1,1){\line(3,2)3}
		\put(1,1){\line(1,2)2}
		\put(5,1){\line(-3,2)3}
		\put(5,1){\line(-1,2)2}
		\put(3,1){\line(0,1)4}
		\put(3,2.35){\oval(2,2.7)}
		\put(.85,.3){$a$}
		\put(2.85,.3){$b$}
		\put(4.85,.3){$c$}
		\put(3.4,2.18){$d$}
		\put(1.35,2.85){$e$}
		\put(4.4,2.85){$f$}
		\put(2.85,5.4){$g$}
		\put(2.2,-.75){{\rm Fig.~6}}
		\put(1,-1.75){{\rm The Fano plane}}
	\end{picture}
\end{center}

\vspace*{10mm}

Put $A:=B\cup B'\cup\{0,1\}$ and let $x,y\in A$. We define $x\le y$ if $x=0$ or $y=1$ or $x=y$ or if $x\in B$ and $y\in B'\setminus\{x'\}$. Then $(A,\le)$ is a poset. If $x$ and $y$ are different elements of $B$ then we define $x\sqcup y:=\big(F(x,y)\big)'$. In all the other cases we define $x\sqcup y:=\max(x,y)$ provided $x$ and $y$ are comparable with each other and $x\sqcup y:=1$ otherwise. If $x$ and $y$ are different elements of $B$ then we define $x'\sqcap y':=F(x,y)$. In all the other cases we define $x\sqcap y:=\min(x,y)$ provided $x$ and $y$ are comparable with each other and $x\sqcap y:=0$ otherwise. Then $(A,\sqcup,\sqcap)$ is a $\lambda$-lattice that is not a lattice. We define the unary operation $'$ on $A$ by the following table:
\[
\begin{array}{l|llllllllllllllll}
x  & 0 & a  & b  & c  & d  & e  & f  & g  & a' & b' & c' & d' & e' & f' & g' & 1 \\
\hline
x' & 1 & a' & b' & c' & d' & e' & f' & g' & a  & b  & c  & d  & e  & f  & g  & 0
\end{array}
\]
Put $\mathbf A:=(A,\sqcup,\sqcap,{}')$. It is easy to see that $'$ is an antitone involution and a complementation and that $\mathbf A$ satisfies the identities $(x\sqcup y)'\approx x'\sqcap y'$ and $(x\sqcap y)'\approx x'\sqcup y'$. Let $h$ and $i$ be different elements of $B$ and put $j:=F(h,i)$. We prove that $\mathbf A$ satisfies the identity
\begin{enumerate}
\item[{\rm(1)}] $(x\sqcup y)\sqcap y'\approx x\sqcap y'$.
\end{enumerate}
If $x,y\in\{0,1\}$ or $y\in\{x,x'\}$ then {\rm(1)} holds.  \\
If $(x,y)=(h,i)$ then $(x\sqcup y)\sqcap y'=(h\sqcup i)\sqcap i'=j'\sqcap i'=h=h\sqcap i'=x\sqcap y'$. \\
If $(x,y)=(h,i')$ then $(x\sqcup y)\sqcap y'=(h\sqcup i')\sqcap i=i'\sqcap i=0=h\sqcap i=x\sqcap y'$. \\
If $(x,y)=(h',i)$ then $(x\sqcup y)\sqcap y'=(h'\sqcup i)\sqcap i'=h'\sqcap i'=j=h'\sqcap i'=x\sqcap y'$. \\
If $(x,y)=(h',i')$ then $(x\sqcup y)\sqcap y'=(h'\sqcup i')\sqcap i=1\sqcap i=i=h'\sqcap i=x\sqcap y'$. \\
By duality we obtain that $\mathbf A$ satisfies the identity
\begin{enumerate}
\item[{\rm(2)}] $(x\sqcap y)\sqcup y'\approx x\sqcup y'$.
\end{enumerate}
Now we have
\begin{align*}
y'\sqcup\big((x\sqcup y')\sqcap y\big) & \approx y'\sqcup(x\sqcap y)\approx x\sqcup y', \\
 \big(x'\sqcup(x\sqcap y)\big)\sqcap x & \approx(x'\sqcup y)\sqcap x\approx x\sqcap y
\end{align*}
showing that $\mathbf A$ satisfies the identities {\rm(C1)} and {\rm(C2)}. Now let $k\in B\setminus\{h,i,j\}$ and put $l:=F(i,k)$ and $m:=F(h,k)$. Then $m=i$ would imply $j=F(h,i)=F(h,m)=k$, a contradiction. Hence $m\ne i$ and $h\le i'$, but
\[
k\sqcup(k'\sqcap h)=k\sqcup h=m'\not\le i'=k\sqcup l=k\sqcup(k'\sqcap i')
\]
showing that $\mathbf A$ does not satisfy condition {\rm(E1)}. By duality, $\mathbf A$ does not satisfy condition {\rm(E2)}. Let $\odot$ and $\to$ denote the Sasaki operations on $A$ defined by {\rm(S2)}. According to Theorem~\ref{th3}, $\odot$ and $\to$ do not form an adjoint pair. Because of {\rm(1)} and {\rm(2)} we have
\begin{align*}
x\odot y & \approx(x\sqcup y')\sqcap y\approx x\sqcap y, \\
  x\to y & \approx x'\sqcup(x\sqcap y)\approx x'\sqcup y.
\end{align*}
Now $h'\odot i'=h'\sqcap i'=j\le k'$, but $h'\not\le k'=i\sqcup k'=i'\to k'$ showing directly that $\odot$ and $\to$ do not satisfy condition {\rm(A1)}. But $\odot$ and $\to$ do not satisfy condition {\rm(A2)}, too, since
\[
k\le j'=h\sqcup i=h\sqcup(h'\sqcap i)=h'\to i,
\]
but
\[
k\odot h'=(k\sqcup h)\sqcap h'=m'\sqcap h'=k\not\le i.
\]
\end{example}

Although the $\lambda$-lattices visualized in Figures~4 and 5 are complemented and the complementation is an antitone involution, the corresponding Sasaki operations defined by (S2) do not form an adjoint pair without regard how the $\sqcup$ and $\sqcap$ are defined. 

However, there exist $\lambda$-lattices with a unary operation not being lattices, but whose Sasaki operations defined by (S2) form an adjoint pair. At first, consider the following example.

\begin{example}\label{ex2}
	Consider the $\lambda$-lattice $\mathbf A=(A,\sqcup,\sqcap,{}')$
	
	\vspace*{-5mm}
	
	\begin{center}
		\setlength{\unitlength}{7mm}
		\begin{picture}(4,8)
			\put(2,1){\circle*{.3}}
			\put(1,3){\circle*{.3}}
			\put(3,3){\circle*{.3}}
			\put(1,5){\circle*{.3}}
			\put(3,5){\circle*{.3}}
			\put(2,7){\circle*{.3}}
			\put(1,3){\line(0,1)2}
			\put(1,3){\line(1,1)2}
			\put(1,3){\line(1,-2)1}
			\put(3,3){\line(-1,-2)1}
			\put(3,3){\line(-1,1)2}
			\put(3,3){\line(0,1)2}
			\put(2,7){\line(-1,-2)1}
			\put(2,7){\line(1,-2)1}
			\put(1.85,.3){$0=c\sqcap d$}
			\put(.35,2.85){$a$}
			\put(3.4,2.85){$b$}
			\put(.35,4.85){$c$}
			\put(3.4,4.85){$d$}
			\put(1.85,7.4){$1=a\sqcup b$}
			\put(1.2,-.75){{\rm Fig.~7}}
			\put(.6,-1.75){A $\lambda$-lattice}
		\end{picture}
	\end{center}
	
	\vspace*{10mm}
	
	with the unary operation $'$ defined by the following table:
	\[
	\begin{array}{l|llllll}
		x  & 0 & a & b & c & d & 1 \\
		\hline
		x' & 1 & b & a & d & c & 0
	\end{array}
	\]
	Then the operation tables of the Sasaki operations $\odot$ and $\to$ on $A$ defined by {\rm(S2)} look as follows:
	\[
	\begin{array}{c|cccccc}
		\odot & 0 & a & b & c & d & 1 \\
		\hline
		0   & 0 & 0 & 0 & 0 & 0 & 0 \\
		a   & 0 & a & 0 & 0 & 0 & a \\
		b   & 0 & 0 & b & 0 & 0 & b \\
		c   & 0 & a & b & c & 0 & c \\
		d   & 0 & a & b & 0 & d & d \\
		1   & 0 & a & b & c & d & 1
	\end{array}
	\quad\quad\quad
	\begin{array}{c|cccccc}
		\to & 0 & a & b & c & d & 1 \\
		\hline
		0  & 1 & 1 & 1 & 1 & 1 & 1 \\
		a  & b & 1 & b & 1 & 1 & 1 \\
		b  & a & a & 1 & 1 & 1 & 1 \\
		c  & d & d & d & 1 & d & 1 \\
		d  & c & c & c & c & 1 & 1 \\
		1  & 0 & a & b & c & d & 1
	\end{array}
	\]
	It can be easily verified that for all $x,y\in A$, $x\odot y$ is the smallest element $z$ of $A$ satisfying $x\le y\to z$. Hence $\odot$ and $\to$ form an adjoint pair. It is interesting and not hard to prove that $\mathbf A$ satisfies conditions {\rm(D1)} and {\rm(D2)}, but according to Theorem~\ref{th5} cannot satisfy both identities {\rm(C1)} and {\rm(C2)}. Indeed, $\mathbf A$ satisfies neither identity {\rm(C1)} nor identity {\rm(C2)} since
	\begin{align*}
	a'\sqcup \big((c\sqcup a')\sqcap a\big) & =b\sqcup\big((c\sqcup b)\sqcap a)=b\sqcup(c\sqcap a)=b\sqcup a=1\ne c=c\sqcup b=c\sqcup a', \\
	\big(c'\sqcup(c\sqcap a)\big)\sqcap c & =(d\sqcup a)\sqcap c=d\sqcap c=0\ne a=c\sqcap a.
	\end{align*}
\end{example}

There exist infinitely many of $\lambda$-lattices the Sasaki operations of which defined by (S2) form an adjoint pair. Namely, for every positive integer $n$ consider the direct power $\mathbf A^n$ of the $\lambda$-lattice $\mathbf A$ (with unary operation) from Example~\ref{ex2}. Since the operations on $\mathbf A^n$ are defined componentwise, every such $\mathbf A^n$ satisfies both conditions (A1) and (A2) and is not a lattice. Moreover, $\mathbf A$ is subdirectly irreducible. Consider the variety $\mathcal V(\mathbf A)$ of $\lambda$-lattices (with unary operation) generated by $\mathbf A$. According to Theorem~4.15 in \cite{CL11} this variety is congruence distributive. Hence the only finite subdirectly irreducible members of this variety are homomorphic images of subalgebras of $\mathbf A$. Because $\mathbf A$ is simple, all finite subdirectly irreducible algebras in $\mathcal V(\mathbf A)$ are subalgebras of $\mathbf A$. Up to $\mathbf A$ itself, these are only the subalgebras with universes $\{0,1\}$, $\{0,a,b,1\}$ and $\{0,c,d,1\}$ which are lattices, in fact Boolean algebras. Because every algebra in $\mathcal V(\mathbf A)$ is a $\lambda$-lattice (with unary operation) which is a subdirect product of subdirectly irreducible members, i.e.\ a subalgebra of a direct product of an arbitrary number of these four algebras, their Sasaki operations form an adjoint pair again.

Of course, the $\lambda$-lattices (with unary operation) mentioned before are not the only $\lambda$-lattices (with unary operations) the Sasaki operations of which form an adjoint pair, other such examples are e.g.\ direct products of $\mathbf A$ with orthomodular lattices (considered as $\lambda$-lattices with a unary operation).

\section{Ordered semirings and ring-like structures}

In this section we investigate so-called {\em ordered semirings with a unary operation}, i.e.\ ordered sixtuples $(S,+,\cdot,0,{}',\le)$ where $(S,+,\cdot,0)$ is a commutative semiring (see e.g.\ \cite G), $'$ a unary operation and $\le$ a partial order relation on $S$ satisfying the identity $xx'\approx0$. Recall from \cite G that a {\em commutative semiring} is an algebra $(S,+,\cdot,0)$ of type $(2,2,0)$ such that the following holds:
\begin{align*}
& (S,+,0)\text{ is a commutative monoid}, \\
& (S,\cdot)\text{ is a commutative semigroup}, \\
& x0\approx0, \\
& \text{the operation }\cdot\text{ is distributive with respect to }+.
\end{align*}
We can transform the Sasaki operations from (S1) by replacing $\vee$ and $\wedge$ by $+$ and $\cdot$, respectively. In this way we obtain the Sasaki operations on $S$ defined by
\begin{enumerate}
	\item[(S3)] $x\odot y:=(x+y')y\quad$ and $\quad x\to y:=x'+xy$
\end{enumerate}
for all $x,y\in S$. Observe that in our case
\[
x\odot y\approx(x+y')y\approx xy+y'y\approx xy+0\approx xy.
\]
We investigate when the Sasaki operations on $S$ defined by (S3) form an adjoint pair.

We introduce the following conditions:
\begin{enumerate}
	\item[(3)] $x\le y'+xyy$,
	\item[(4)] $x\le y$ implies $z'+zx\le z'+zy$,
	\item[(5)] $x\le y$ implies $xz\le yz$,
	\item[(6)] $xy\le x$.
\end{enumerate}

Using of these conditions, we can state and prove the following result.

\begin{theorem}\label{th1}
Let $\mathbf S=(S,+,\cdot,0,{}',\le)$ be an ordered semiring with a unary operation and $\odot$ and $\to$ denote the Sasaki operations on $S$ defined by {\rm(S3)}. Then the following holds:
\begin{enumerate}[{\rm(i)}]
\item If $\mathbf S$ satisfies conditions {\rm(3)} and {\rm(4)} then $\odot$ and $\to$ satisfy condition {\rm(A1)},
\item if $\mathbf S$ satisfies conditions {\rm(5)} and {\rm(6)} then $\odot$ and $\to$ satisfy condition {\rm(A2)},
\item if $\mathbf S$ satisfies conditions {\rm(3)} -- {\rm(6)} then $\odot$ and $\to$ form an adjoint pair.
\end{enumerate}
\end{theorem}

\begin{proof}
Let $a,b,c\in S$.
\begin{enumerate}[(i)]
\item If $\mathbf S$ satisfies conditions (3) and (4) and $a\odot b\le c$ then we obtain
\[
a\le b'+abb=b'+b(ab)=b'+b(a\odot b)\le b'+bc=b\to c.
\]
\item If $\mathbf S$ satisfies conditions (5) and (6) and $a\le b\to c$ then we obtain
\[
a\odot b=ab\le(b\to c)b=(b'+bc)b=b'b+bcb=0+c(bb)=c(bb)\le c.
\]
\item This follows from (i) and (ii).
\end{enumerate}
\end{proof}

\begin{example}
	Consider a Boolean ring $\mathbf R=(R,+,\cdot,0,1)$ and define a unary operation $'$ and a binary relation $\le$ on $R$ by $x':=x+1$ and $x\le y$ whenever $xy=x$ for all $x,y\in R$, respectively. Then $\mathbf R:=(R,+,\cdot,0,{}',\le)$ is an ordered semiring with a unary operation satisfying conditions {\rm(3)} -- {\rm(6)}. Namely, let $a,b,c\in R$. Then we have:
	\begin{align*}
		& aa'=a\wedge a'=0, \\
		& a\le b'\vee a=(b\wedge a')'=(ba')'=b(a+1)+1=b+1+ab=b'+abb, \\
		& a\le b\text{ implies }c'+ca=c+1+ca=c(a+1)+1=(c\wedge a')'\le(c\wedge b')'=c(b+1)+1= \\
		& \hspace*{24mm}=c+1+cb=c'+cb, \\
		& a\le b\text{ implies }ac=a\wedge c\le b\wedge c=bc, \\
		& ab=a\wedge b\le a.
	\end{align*}
According to Theorem~\ref{th1} {\rm(iii)} the Sasaki operations on a Boolean ring $\mathbf R$ defined by {\rm(S3)} form an adjoint pair.
\end{example}

Finally, we are interested in algebras similar to semirings in which also Sasaki operations forming an adjoint pair can be defined.

In \cite C the first author introduced the following notion:

\begin{definition}
An {\em orthomodular pseudoring} is an algebra $(R,+,\cdot,0,1)$ of type $(2,2,0,0)$ such that $(A,+,0)$ is a commutative groupoid with neutral element $0$, $(A,\cdot,1)$ is a semilattice with neutral element $1$ and the following identities are satisfied:
\begin{align*}
                                 x+x & \approx0, \\
                                  x0 & \approx 0, \\
                             (x+1)+y & \approx x+(1+y), \\
                             (1+xy)x & \approx x+xyx, \\
                         (1+x)(1+xy) & \approx 1+x, \\
\big(1+x(1+y)\big)\big(1+y(1+x)\big) & \approx 1+(x+y), \\
                           (x+xy)+xy & \approx x.
\end{align*}
\end{definition}

Orthomodular pseudorings are closely related to orthomodular lattices in a similar way as Boolean rings are related to Boolean algebras.

\begin{theorem}\label{th2}
{\rm(}cf.\ {\rm\cite C)} If $(L,\vee,\wedge,{}',0,1)$ is an orthomodular lattice and
\begin{align*}
x+y & :=(x\wedge y')\vee(x'\wedge y), \\
 xy & :=x\wedge y
\end{align*}
for all $x,y\in L$ then $(L,+,\cdot,0,1)$ is an orthomodular pseudoring. If, conversely, $(R,+,\cdot,0,1)$ is an orthomodular pseudoring and
\begin{align*}
  x\vee y & :=1+(1+x)(1+y), \\
x\wedge y & :=xy, \\
       x' & :=1+x
\end{align*}
for all $x,y\in R$ then $(R,\vee,\wedge,{}',0,1)$ is an orthomodular lattice. This correspondence between orthomodular lattices and orthomodular pseudorings is one-to-one.
\end{theorem}

We can translate the Sasaki operations defined by (S1) for orthomodular lattices into the operations $+$ and $\cdot$ of the corresponding orthomodular pseudoring $(R,+,\cdot,0,1)$ as follows:
\begin{enumerate}
\item [(S4)] $x\odot y:=\big(1+(1+x)y\big)y\quad$ and $\quad x\to y:=1+x(1+xy)$
\end{enumerate}
for all $x,y\in R$.

\begin{example}
	Consider the following orthomodular pseudoring $(R,+,\cdot,{}',0,1)$ with $R=\{0,a,b,c,d,1\}$ and
	\[
	\begin{array}{c|cccccc}
		+ & 0 & a & b & c & d & 1 \\
		\hline
		0 & 0 & a & b & c & d & 1 \\
		a & a & 0 & 0 & 1 & 0 & c \\
		b & b & 0 & 0 & 0 & 1 & d \\
		c & c & 1 & 0 & 0 & 0 & a \\
		d & d & 0 & 1 & 0 & 0 & b \\
		1 & 1 & c & d & a & b & 0
	\end{array}
	\quad
	\begin{array}{c|cccccc}
		\cdot & 0 & a & b & c & d & 1 \\
		\hline
		0   & 0 & 0 & 0 & 0 & 0 & 0 \\
		a   & 0 & a & 0 & 0 & 0 & a \\
		b   & 0 & 0 & b & 0 & 0 & b \\
		c   & 0 & 0 & 0 & c & 0 & c \\
		d   & 0 & 0 & 0 & 0 & d & d \\
		1   & 0 & a & b & c & d & 1
	\end{array}
	\quad
	\begin{array}{l|l}
		x & x' \\
		\hline
		0 & 1 \\
		a & c \\
		b & d \\
		c & a \\
		d & b \\
		1 & 0
	\end{array}
	\]
	The Sasaki operations on $R$ defined by {\rm(S4)} read as follows:
	\[
	\begin{array}{c|cccccc}
		\odot & 0 & a & b & c & d & 1 \\
		\hline
		0   & 0 & 0 & 0 & 0 & 0 & 0 \\
		a   & 0 & a & b & 0 & d & a \\
		b   & 0 & a & b & c & 0 & b \\
		c   & 0 & 0 & b & c & d & c \\
		d   & 0 & a & 0 & c & d & d \\
		1   & 0 & a & b & c & d & 1
	\end{array}
	\quad
	\begin{array}{c|cccccc}
		\to & 0 & a & b & c & d & 1 \\
		\hline
		0  & 1 & 1 & 1 & 1 & 1 & 1 \\
		a  & c & 1 & c & c & c & 1 \\
		b  & d & d & 1 & d & d & 1 \\
		c  & a & a & a & 1 & a & 1 \\
		d  & b & b & b & b & 1 & 1 \\
		1  & 0 & a & b & c & d & 1
	\end{array}
	\]
\end{example}

Using Theorem~\ref{th2} and the result in \cite{CL17} on orthomodular lattices mentioned in the beginning we immediately obtain

\begin{theorem}
If $(R,+,\cdot,0,1)$ is an orthomodular pseudoring then the Sasaki operations on $R$ defined by {\rm(S4)} form an adjoint pair.
\end{theorem}




{\bf Data availability statement} No datasets were generated or analyzed during the current study.



{\bf Affiliations}

{\bf Ivan Chajda}$^1$ $\cdot$ {\bf Helmut L\"anger}$^{1,2}$

Ivan Chajda \\
ivan.chajda@upol.cz

$^1$ Faculty of Science, Department of Algebra and Geometry, Palack\'y University Olomouc, 17.\ listopadu 12, 771 46 Olomouc, Czech Republic

$^2$ Faculty of Mathematics and Geoinformation, Institute of Discrete Mathematics and Geometry, TU Wien, Wiedner Hauptstra\ss e 8-10, 1040 Vienna, Austria
\end{document}